\documentclass[12pt,a4paper]{article}
\usepackage[utf8]{inputenc}
\usepackage{amsfonts}
\usepackage{amsthm}
\usepackage{amsmath}
\usepackage{amssymb}
\usepackage{todonotes}
\usepackage{graphicx}
\newtheorem{trm}{Theorem}

\newtheorem{lem}[trm]{Lemma}
\newtheorem{cla}[trm]{Claim}

\begin{document}

\title{Maximal edge colorings of graphs}

\author{Sebastian Babiński\thanks{Faculty of Mathematics and Computer Science, Jagiellonian University, {\L}ojasie\-wicza~6, 30-348 Krak\'{o}w, Poland. E-mail: {\tt Sebastian.Babinski@im.uj.edu.pl}.}\and
Andrzej Grzesik\thanks{Faculty of Mathematics and Computer Science, Jagiellonian University, {\L}ojasie\-wicza~6, 30-348 Krak\'{o}w, Poland. E-mail: {\tt Andrzej.Grzesik@uj.edu.pl}. The work of this author was supported by the National Science Centre grant 2016/21/D/ST1/00998.}}

\date{}

\maketitle

\begin{abstract}
For a graph $G$ of order $n$ a maximal edge coloring is a proper edge coloring with $\chi'(K_n)$ colors such that adding any edge to $G$ in any color makes it improper. 
Meszka and Tyniec proved that for some values of the number of edges there are no graphs with a maximal edge coloring, while for some other values, they provided constructions of such graphs. 
However, for many values of the number of edges determining whether there exists any graph with a maximal edge coloring remained open. 
We give a complete solution of this problem.
\end{abstract}

\section{Introduction}

The problems associated with maximality of a family of some objects are widely considered in Combinatorics. 
When we have a subset of the space of objects that satisfies some fixed properties we say it is \textit{maximal} when adding to this set any new element from the space results in violating the conditions. 
One of the best-known problem in this field is related with maximal partial latin squares. A \textit{partial latin square} is an array with $n$ rows and $n$ columns such that each of its entries is filled with a number from $1$ to $n$ or left empty and each number can appear in every column and every row at most once. For an exhaustive survey on latin squares see \cite{DK}. We say that a partial latin square is \textit{maximal} when filling any empty cell with any number from $1$ to $n$ results in repetition of this number in the row or in the column. The problem is to determine for which values of $m$ there exists a maximal partial latin square with exactly $m$ cells filled. This problem was partially solved by Horák and Rosa in \cite{HR}, but there are still many values of $m$, for which the problem is open. 


Many problems of a similar spirit have already been investigated. In particular, Cousins and Wallis in \cite{CW} partially answered the problem of determinig the possible order of a maximal set of pairwise edge-disjoint 1-factors of $K_{2n}$, but the complete answer was given by Reed and Wallis in \cite{RW} more than 15 years later. An analogous problem of determining the possible order of a maximal set of edge-disjoint 2-factors was solved by Hoffman, Rodger and Rosa in~\cite{HRR}. They also proved that the answer for this problem is the same if we consider not any 2-factors, but just Hamiltonian cycles. A similar problem for triangle-factors of $K_{3n}$ was considered by Rees, Rosa and Wallis in \cite{RRW}. A generalization of the latin squares problem in higher dimension -- for partial latin cubes -- was investigated by Britz, Cavenagh and Sørense in~\cite{BCS}. These are only a few examples of the huge family of similar problems. For an extended survey on this topic, see \cite{R15}.
 
Comming back to the problem of the maximal partial latin squares, it is easy to notice that the task of determining how many non-empty cells a maximal partial latin square can have is equivalent to the problem of determining how many edges can be in the maximal partial edge coloring of a complete bipartite graph $K_{n,n}$ with $n = \chi'(K_{n,n})$ colors. As a natural consequence of the above problem, and motivated by the developement of the theory of \hbox{on-line} maximal edge coloring (see \cite{BN}, \cite{FN}), Meszka and Tyniec in~\cite{MT} analyzed maximal partial colorings of a complete graph $K_n$ with $\chi'(K_n)$ colors. Notice that the chromatic index of a complete graph on $n \ge2$ vertices is equal to $n-1$ when $n$ is even, and is equal to $n$ when $n$ is odd. 

We say that a fixed graph $G$ of order $n$ has a \textit{maximal edge coloring} if there exists a proper edge coloring of $G$ with $\chi'(K_n)$ colors such that adding to the graph $G$ any additional edge in any color will make the coloring improper. Obviously, not every graph has maximal edge coloring. That is why for any positive integer $n$ we define a \textit{spectrum} (denoted by $MEC(n)$) as the set of all values $m$ of the number of edges for which there exists a graph of order $n$ and size $m$ that has a maximal edge coloring:
\begin{center}
$MEC(n) = \{m \in \mathbb{N}:$ there exists a graph $G$ such that $|V(G)| = n$\\
and $|E(G)| = m$ which has a maximal edge coloring $\}$.
\end{center}

Meszka and Tyniec in \cite{MT} proved the following two theorems. 

\begin{trm}
Let $n$ be an even number, $n > 10$.
\begin{itemize}
\item If $\frac{1}{4}n^2 \le m \le \frac{1}{2}n^2 - \frac{1}{2}n = \binom{n}{2}$, $m \ne \binom{n}{2}-1$, and for $n \equiv 2\ (\textrm{mod } 4)$, $m \ne \frac{1}{4}n^2 + 1$, then $m \in MEC(n)$.
\item If $0 \le m \le \frac{1}{4}n^2 - \frac{3}{8}n$ or $m = \binom{n}{2} - 1$, then $m \notin MEC(n)$.
\end{itemize}
\end{trm}

\begin{trm}
Let $n$ be an odd number, $n > 10$.
\begin{itemize}
\item If $\frac{1}{4}n^2 + \frac{1}{2}n - \frac{3}{4} \le m \le \frac{1}{2}n^2 - \frac{1}{2}n = \binom{n}{2}$, then $m \in MEC(n)$.
\item If $0 \le m < \frac{1}{4}n^2 - \frac{1}{4}n$, then $m \notin MEC(n)$.
\end{itemize}
\end{trm}

Moreover, by computer analysis $MEC(n)$ for $3 \le n \le 10$ were completely determined.

From the above theorems it is clearly seen that the problem has not been solved for $\frac{1}{4}n^2 - \frac{3}{8}n < m \le \frac{1}{4}n^2 - 1$ when $n \ge 12$ and even, for $m = \frac{1}{4}n^2+1$ when $n \ge 14$, $n \equiv 2$ (mod $4$) and for $\frac{1}{4}n^2 - \frac{1}{4}n \le m \le \frac{1}{4}n^2 + \frac{1}{2}n - \frac{7}{4}$ when $n \ge 11$ and odd.

In this paper we solve this problem completely for all $n$ by proving the following three theorems.

\begin{trm}\label{thm:even}
If $n\ge4$ is even and $m \in MEC(n)$, then $m \ge \frac{1}{4}n^2$.
\end{trm}

\begin{trm}\label{thm:even2}
If $n\ge10$ and $n \equiv 2$ (mod $4$), then $\frac{1}{4}n^2 + 1 \notin MEC(n)$.
\end{trm}

\begin{trm}\label{thm:odd}
If $n\ge9$ is odd and $m \in MEC(n)$, then $m \ge \frac{1}{4}n^2 + \frac{1}{2}n - \frac{3}{4}$.
\end{trm}

In order to prove Theorem~\ref{thm:even} we compute a bound on the number of edges that the sum of degrees of non-adjacent vertices implies, and consider a structure of a graph
that yields a maximal edge-coloring. To prove Theorem~\ref{thm:even2} it is essential to analyze
the structure of a graph and the distribution of colors. These proofs are in Section \ref{sec:even}. The proof of Theorem~\ref{thm:odd} uses similar techniques but is more complicated because of bigger number of cases and more complex structures that can appear. It is considered in Section \ref{sec:odd}. 

\section{The even order case}\label{sec:even}

Let $d(v)$ denote the degree of the vertex $v$.
We say that a vertex \textit{can see} some color $c$ if there exists an edge in color $c$ incident to this vertex. If a graph of order $n$ has a maximal edge coloring, then any not connected vertices $u$ and $v$ must see together all the colors, because otherwise there is a color $c$ that neither $u$ nor $v$ can see, and so we can add an edge $uv$ colored~$c$. For even $n$ we have $\chi'(K_n) = n-1$ and so we have $n-1$ colors and any not connected vertices $u$ and $v$ must satisfy $d(u) + d(v) \ge n-1.$ 

We start with proving the following lemma. 

\begin{lem}\label{lem:even_n}
Let $G$ be a graph of even order $n$ and size $m$. If every two not connected vertices $u, v \in V(G)$ satisfy $d(u) + d(v) \ge n$, then $m \ge \frac{1}{4}n^2$.
\end{lem}

\begin{proof}
Let us sum up the inequality $d(u) + d(v) \ge n$ over all $uv \notin E(G)$. On the right-hand side we simply get $n$ times the number of non-edges in the graph~$G$. On the left-hand side each $d(u)$ is counted the number of times that is equal to the non-degree of the vertex $u$, so the left-hand side is equal to 
$$\sum_{u \in V} d(u)(n-1-d(u)) = (n-1)\sum_{u \in V} d(u) - \sum_{u \in V} d^2(u) = 2(n-1)m - \sum_{u \in V} d^2(u).$$

We find a lower bound for the sum of squares of degrees, using the inequality between the quadratic and arithmetic mean:
$$\sum_{u \in V} d^2(u) = n \sqrt{\frac{\sum_{u \in V} d^2(u)}{n}}^2 \ge n \left( \frac{\sum_{u \in V} d(u)}{n} \right)^2 = n  \left(\frac{2m}{n} \right)^2 = \frac{4m^2}{n}.$$
Altogether it gives 
$$2(n-1)m - \frac{4m^2}{n} \ge n\left(\binom{n}{2}-m\right),$$
which evaluates to 
$$8m^2 + 2n(2-3n)m + n^4 - n^3 \le 0.$$

This is a quadratic inequality on variable $m$, and so $m$ is greater than or equal to the smaller root. It implies $m\ge\frac{1}{4}n^2$, as wanted. 
\end{proof}

We are now ready to prove Theorem~\ref{thm:even}.

\begin{proof}[Proof of Theorem~\ref{thm:even}]
Let $G$ be a graph of size $n$ and order $m < \frac{1}{4}n^2$ having a maximal edge coloring. Since $n$ is even, $G$ is colored using a set $C$ of $n-1$ colors. 
From Lemma~\ref{lem:even_n} we can assume that there exist two vertices $u$ and $v$ which are not connected such that $d(u) + d(v) = n-1$. Together they must see all the colors, so if $A \subset C$ is the set of all colors that $u$ can see, then $v$ can see exactly colors in $A' = C \setminus A$. 

Notice that every vertex of the graph $G$ needs to be connected with $u$ or $v$. Indeed, if there exists some vertex $t$ connected neither with $u$ nor with~$v$, then $t$ would have to see all $n-1$ colors, which is not possible as it can be only connected with the remaining $n-3$ vertices. Let us denote by $W$ the set of common neighbors of $u$ and $v$. Then 
$$n-2 = |A| + |A'| - |W| = n-1 - |W|,$$
which means that $u$ and $v$ have exactly one common neighbor,  which will be denoted by $w$. It also means that every vertex in $V(G) \setminus \{u, v, w\}$ is connected exactly with one of the vertices $u$ or $v$. 

Let $L$ be the set of all vertices except $v$ and its neighbors and $R$ the set of all vertices except $u$ and its neighbors. Denoting $a=|A|$ we have $|A'| = n-a-1$, $|L|=a$ and $|R|=n-a-1$. This is depicted on Figure~\ref{fig:even}.

\begin{figure}[ht]
\begin{center}
\includegraphics[scale=1.2]{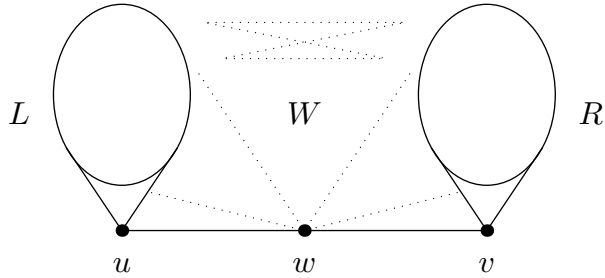}
\end{center}
\caption{A graph of even order $n$ with two vertices $u$ and $v$ having the sum of the degrees equal to $n$. Dotted lines denote the possible edges.}\label{fig:even}
\end{figure}

Let us estimate the sum of the degrees of all the vertices in $G$. Since none of the vertices in $L$ can be connected with $v$, each of them must see all the colors from $A$. This means that it is connected by an edge in color from $A$ with at least one vertex in $R\cup\{w\}$. Similarly, each vertex in $R$ must see all the colors from $A'$ and is connected by an edge in color from $A'$ with $w$ or some vertex in~$L$. Thus the sum of all the degrees is at least 
$$a^2 + (n-a-1)^2 + a + n-a-1 = a^2 + (n-a-1)^2 + n - 1.$$
The last expression is minimized for $a = \frac{n}{2}$ (as $\frac{n-1}{2}$ is not an integer), so the sum of the degrees is at least $\frac{1}{2}n^2$, which contradicts the assumed number of edges. 
\end{proof}

In the proof of Theorem \ref{thm:even2} we also consider similar two cases, but this time we need to use some additional arguments. 

\begin{proof}[Proof of Theorem \ref{thm:even2}]
Let $G$ be a graph of order $n \ge10$ for \hbox{$n \equiv 2$ (mod~4)} with $\frac{1}{4}n^2+1$ edges, that has a maximal edge coloring using a set $C$ of $n-1$ colors. We know that the sum of the degrees of any two non-adjacent vertices is at least $n-1$. 

Firstly, assume that there exists a pair of vertices $u, v \in V$ such that $d(u) + d(v) = n-1$. 
Using observations made in the proof of Theorem~\ref{thm:even}, vertices $u$ and $v$ have exactly one common neighbor $w$, and if $A$ is the set of colors that $u$ can see, then all its neighbors except $w$ need to see all the colors from $A$. Also, the set of colors that $v$ can see is exactly $A' = C \setminus A$, and all its neighbors except $w$ need to see all the colors from $A'$. Let us denote by $L$ the set of all vertices except $v$ and its neighbors and by $R$ the set of all vertices except $u$ and its neighbors. Denoting $a=|A|$ we have $|L|=a$ and $|R|=n-a-1$. This is depicted on Figure \ref{fig:even}.

As in the proof of Theorem~\ref{thm:even}, every vertex in $L$ can have in $L$ at most $a-1$ edges in colors from $A$ and so it is connected by an edge in color from $A$ to at least one vertex in $R\cup \{w\}$. Counting similarly for vertices in $R$, we obtain that the sum of the degrees is at least
\begin{equation}\label{eq:even2degreesum}
a^2 + (n-a-1)^2 + a + n-a-1.
\end{equation}

If $a \ge \frac{1}{2}n+1$ or $a \le \frac{1}{2}n-2$, then this is at least $\frac{1}{2}n^2+4$, contradicting the assumed number of edges, and so $a = \frac{1}{2}n$ or $a = \frac{1}{2}n-1$. From the symmetry, we may assume that $a=\frac{1}{2}n$. For such $a$, the bound in (\ref{eq:even2degreesum}) gives that the sum of the degrees is at least $\frac{1}{2}n^2$.

If there exist vertices $x, y \in L$ that are not connected by an edge in color from $A$, then each of them must have one more neighbor in $R\cup \{w\}$ connected by an edge in color from $A$, which increases the bound in (\ref{eq:even2degreesum}) by $2$. To avoid contradiction with the number of edges, $x$ and $y$ cannot be connected. But then, they need to see all the colors from $A'$, which gives (together with the edge $vw$) more edges in colors from $A'$ outside $R$ then it was counted in $(\ref{eq:even2degreesum})$. Thus, all the vertices in $L$ are connected by edges in colors from $A$. Similarly, all the vertices in $R$ are connected by edges in colors from $A'$. 

This also means that each vertex in $L$ has exactly one neighbor in $R\cup \{w\}$ connected by an edge in color from $A$ and each vertex in $R$ has exactly one neighbor in $L\cup\{w\}$ connected by an edge in color from $A'$. Counting those edges together with edges inside $L$ and inside $R$, we have $\frac{1}{4}n^2$ edges. The only edge not counted yet needs to be between $L$ and $w$ in color from $A'$, or between $R$ and $w$ in color from $A'$. 

Since $|R|=\frac{1}{2}n-1 \ge 4$ is even and $R$ is a complete graph with edges in colors from $A'$, there are at least two colors in $A'$ with $|R|/2$ edges inside~$R$. At least one of those colors cannot be seen by vertices in $L$ and by $w$, and so $L\cup\{w\}$ is a complete graph. If all the edges between $L$ and $w$ are in colors from $A$, then there needs to be an edge in color from $A$ between $R$ and $w$, but then $w$ would have more edges in colors from $A$ than the size of~$A$. Thus, $w$ is connected with $|L|-1$ vertices in $L$ by edges in colors from $A$ and with one vertex in $L$ by an edge in color from $A'$. 

This means that there exists a color $c\in A$ that appears on $(|L|-1)/2$ edges in $L\cup\{w\}$ and all other colors from $A$ appear on $(|L|+1)/2$ edges in $L\cup\{w\}$. Color $c$ is not seen by vertex $w$ and by all vertices in $R$ except one. This means that $w$ is connected with all vertices in $R$ but one. In particular, $w$ sees all the colors from $A'$, which contradicts the statement proven in the previous paragraph that there is a color in $A'$ not seen by $w$.  

Let us consider the second possible case, i.e., when the sum of the degrees of any two non-adjacent vertices is at least $n$. Summing up this assumption over all pairs of non-edges and reducing as in the proof of Lemma \ref{lem:even_n}, we get
$$ 2(n-1)m - \sum_{u\in V} d^2(u) \ge n\left(\binom{n}{2}-m\right).$$
Since $m=\frac{1}{4}n^2+1$, we obtain that
\begin{equation}\label{eq:even2degree}
\sum_{u \in V} d^2(u) \le \frac{1}{4}n^3 + 3n - 2.
\end{equation}

If there exists a color seen by at most $\frac{1}{2}n - 3$ vertices, then all the remaining vertices form a clique. It means that there are at least $\frac{1}{2}n + 3$ vertices that have degree at least $\frac{1}{2}n + 2$. Since the sum of the degrees equals to $\frac{1}{2}n^2+2$, the minimum sum of squares of the degrees is obtained when the degrees are distributed as equally as possible. Thus, in the potential minimum (not necessarily achievable), the remaining $\frac{1}{2}n - 3$ vertices have degrees $\frac{1}{2}n - 3$ or $\frac{1}{2}n - 2$, and the sum of squares of the degrees is at least $$\left(\frac{n}{2} + 3\right)\left(\frac{n}{2} + 2\right)^2 + \left(\frac{n}{2} - 13\right)\left(\frac{n}{2} - 2\right)^2 + 10\left(\frac{n}{2} - 3\right)^2 = \frac{1}{4}n^3 + 6n+50,$$ 
which contradicts $(\ref{eq:even2degree})$.

If every color is seen by at least $\frac{1}{2}n+1$ vertices, then we obtain a contradiction with the total number of edges, and so there exists a color that is seen exactly by $\frac{1}{2}n-1$ vertices. It means that there exists a clique $K$ of size $\frac{1}{2}n + 1$.

Notice that if there are at least $\frac{1}{2}n+1$ edges between $K$ and the rest of the graph, then the minimum sum of squares of the degrees is at least $$\left(\frac{n}{2} + 1\right)\left(\frac{n}{2} + 1\right)^2 + \left(\frac{n}{2} - 1\right)\left(\frac{n}{2} - 1\right)^2 = \frac{1}{4}n^3 + 3n,$$ 
which contradicts $(\ref{eq:even2degree})$. Thus, there must be a vertex $v \in K$, which is not connected to any vertex outside $K$. Since $v$ has degree exactly $\frac{1}{2}n$ and any two not connected vertices have the sum of the degrees at least $n$, each vertex outside $K$ has degree at least $\frac{1}{2}n$. It means that each such vertex has at least $2$ neighbors in $K$. This creates $n-2$ edges between $K$ and the rest of the graph, which gives a contradiction. 
\end{proof}

\section{The odd order case}\label{sec:odd}

Now we consider the case when the order of a graph is odd. Recall that in this case the number of colors is equal to the order of the graph. 

We start with a lemma which proof is very similar to the proof of the Lemma~\ref{lem:even_n}.

\begin{lem}\label{lem:odd_n+2}
Let $G$ be a graph of odd order $n$ and size $m$. If every two not connected vertices $u, v \in V(G)$ satisfy $d(u) + d(v) \ge n+2$, then \hbox{$m \ge \frac{1}{4}n^2 + \frac{1}{2}n - \frac{3}{4}$}.
\end{lem}

\begin{proof}
Let us sum up the inequality $d(u) + d(v) \ge n+2$ over all $uv \notin E(G)$. Estimating the left-hand side exactly the same way as it was done in the proof of Lemma \ref{lem:even_n} we obtain that it is at most $2m(n-1) - \frac{4m^2}{n}$.
The right-hand side is $n+2$ times the number of non-edges, thus we obtain the inequality
$$2m(n-1) - \frac{4m^2}{n} \ge (n+2)\left(\binom{n}{2} - m\right),$$ 
which can be written as
$$8m^2 - 6mn^2 + n^4+n^3 - 2n^2 \le 0.$$

As this is a quadratic inequality on variable $m$, it implies that $m$ is greater than or equal to the smaller root. Hence we obtain
$m \ge \frac{1}{4}n^2 + \frac{1}{2}n - \frac{3}{4},$
as desired.
\end{proof}

Using this lemma we can prove Theorem~\ref{thm:odd}. 

\begin{proof}[Proof of Theorem~\ref{thm:odd}]
Let $G$ be a graph of an odd order $n$ at least $9$ and size $m < \frac{1}{4}n^2 + \frac{1}{2}n - \frac{3}{4}$ that has a maximal edge coloring and let $C$ be the set of colors ($|C|=n$). The proof consists of a series of claims leading to a contradiction. 

If two vertices $u, v \in V(G)$ are not connected by an edge, they must see all $n$ colors, which means that $d(u) + d(v) \ge n$. We will show that there cannot exist two not connected vertices $u$ and $v$ having $d(u) + d(v) = n$.

\begin{cla}\label{cla:odd_n+1}
If $u,v \in V(G)$ are not connected then $d(u) + d(v) \ge n+1$.
\end{cla}

\begin{proof}
Assume the contrary. Since the sum of the degrees of $u$ and $v$ equals to the number of colors and together $u$ and $v$ see all the colors, the sets of colors seen by $u$ and $v$ are disjoint. Let $A$ be the set of colors that $u$ can see and $A' = C \setminus A$ be the set of colors that $v$ can see. Every vertex must be connected with $u$ or $v$, because otherwise it would have to see all the colors, which is not possible. Let us denote by $W$ the set of common neighbors of $u$ and $v$. Then $n - 2 = |A| + |A'| - |W| = n - |W|,$
which means that $u$ and $v$ have exactly two common neighbors. 

Let us denote by $L$ the set of all vertices except $v$ and its neighbors and by $R$ the set of all vertices except $u$ and its neighbors. Denoting $a = |A|$ we get $|A'| = n - a$, $|L| = a-1$ and $|R| = n-a-1$. This is depicted on Figure~\ref{fig:odd_n}.

\begin{figure}[ht]
\begin{center}
\includegraphics[scale=1.2]{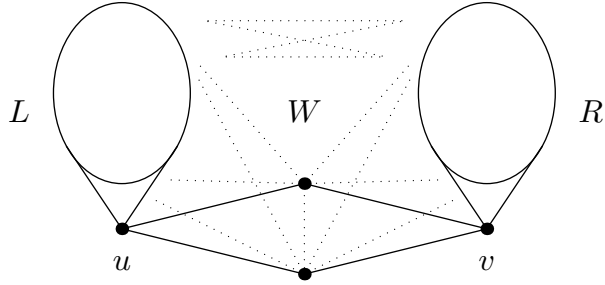}
\end{center}
\caption{A graph of odd order $n$ with two vertices $u$ and $v$ having the sum of the degrees equal to $n$. Dotted lines denote the possible edges.}\label{fig:odd_n}
\end{figure}

Now we estimate the sum of the degrees of all the vertices in $G$. Since each vertex in $L$ is not connected with $v$, it needs to see all the colors from~$A$. And because $|L| = a-1$, each vertex in $L$ is connected by edges in colors from $A$ with at least two vertices from $R \cup W$. Similarly, each vertex in $R$ needs to see all the colors from $A'$ and each vertex from $R$ is connected by edges in colors from $A'$ with at least two vertices from $L \cup W$. Thus, the sum of the degrees of all the vertices is at least
$$a(a-1) + (n-a)(n-a-1) + 2(a-1) + 2(n-a-1) = a^2 + (n-a)^2 + n - 4.$$
The last expression is minimized for $a = \frac{n-1}{2}$ (as $\frac{n}{2}$ is not an integer), so the sum of degrees is at least $\frac{1}{2}n^2 + n - \frac{7}{2}$. To get a contradiction with the assumed number of edges we just need one additional edge that has not been counted yet.

Let us notice that exactly one of the numbers $a - 1$ or $n-a-1$ is even. Thus, without loss of generality, we can assume that $a-1$ is even. If there exist two vertices $x, y \in L$ such that $xy \notin E(G)$ or the edge $xy$ is in color from $A'$, then both $x$ and $y$ need to have at least three edges in colors from~$A$ that connect them with some vertices from $R \cup W$, which increases the sum of the degrees by $2$. So we can assume that $L$ forms a clique with all the edges in colors from $A$. 

The average number of edges in $L$ per color from $A$ equals to 
$$\frac{1}{a} \cdot \binom{a-1}{2} = \frac{a-3}{2} + \frac{1}{a} > \frac{a-3}{2}.$$
Since $\frac{a-3}{2}$ is an integer, there exists a color $d \in A$ with $\frac{a-3}{2}+1 = \frac{a-1}{2}$ edges in  $L$ colored. Thus, there is a perfect matching in color $d$. 

It means that both vertices in $W$ do not see color $d$, which means that they are connected, or at least one of them has color $d$ on an edge connecting it with a vertex in $R$. In both cases we have an additional edge not counted previously, which leads to contradiction. 
\end{proof}

If there are three vertices forming an independent set of size $3$, then each color needs to be seen by at least two of those vertices, so their sum of the degrees is at least $2n$. We show that such case is not possible by proving even a slightly stronger statement. 

\begin{cla}\label{cla:odd_2n-1}
There are no three vertices $v_1$, $v_2$ and $v_3$ with at least one non-edge between them and $d(v_1)+d(v_2)+d(v_3) \ge 2n-1$.
\end{cla}

\begin{proof}
Assume there are such vertices $v_1$, $v_2$ and $v_3$ forming a set~$I$. 
The sum of the degrees of all the vertices of the graph is strictly smaller than $\frac12 n^2 + n - \frac32$, and since it needs to be even, it is at most $\frac12 n^2 + n - \frac72$.

    If the minimum degree $\delta$ is at least $\frac{n+1}{2}$, then the sum of all the degrees is at least
$$2n-1 + (n-3)\delta \ge \frac12 n^2 + n - \frac52,$$
    contradicting the assumed number of edges. Thus, $\delta \le \frac{n-1}{2}$.

Now, take any vertex $u$ of minimum degree $\delta$ and let $S$ be the set of non-neighbors of $u$. All the vertices in $S$ have degree at least $n+1-\delta$ from Claim~\ref{cla:odd_n+1}. 
If at least two vertices from $I$ are in $N(u)$, then the sum of the degrees is at least
$$(\delta-1)\delta + (n-2-\delta)(n+1-\delta) + 2n-1,$$
which evaluates to
$$2\delta^2-2n\delta+n^2+n-3.$$
    This is minimized for $\delta=\frac{n-1}{2}$ and equals to $\frac12 n^2 + n - \frac52$, contradicting the assumed number of edges.

If all three vertices from $I$ are in $S$, then since between them is at least one non-edge, there are two vertices in $I$ forming an independent set of size $3$ with the vertex $u$ and we can replace the remaining vertex of $I$ with $u$. So, we can assume that exactly two vertices from $I$ are in $S$. The sum of the degrees is at least
$$\delta^2+ (n-3-\delta)(n+1-\delta) + 2n-1 = 2\delta^2-(2n-2)\delta+n^2-4.$$
    If $\delta \le \frac{n-3}{2}$, then this is at least $\frac12 n^2 + n - \frac52$, contradicting the assumed number of edges.
    Thus, $\delta=\frac{n-1}{2}$ and the above lover-bound for the sum of the degrees is equal to $\frac12 n^2 + n - \frac92$.
We only need to show that we can enlarge it by a term at least $2$ that was not counted in the above estimate. 

If there are two not connected vertices in $N(u)$, then their sum of the degrees is at least $n+1$, while we were bounding it by $2\delta=n-1$, which gives the missing term. Thus, each vertex in $N(u)$ is connected with all other vertices in $N(u)$ and with $u$. This gives already $\delta$ edges from each vertex and so each edge between $N(u)$ and $S$ enlarges the computed lower-bound. On the other hand, each vertex in $S$ has degree at least $n+1-\delta = |S| + 2$, and so there are many edges between $N(u)$ and $S$, contradiction.
\end{proof}

Notice now, that if we denote by $W$ the set of common neighbors of any two not connected vertices $u$ and $v$, then we have from Claim~\ref{cla:odd_n+1} and Claim~\ref{cla:odd_2n-1}
$$n-2 = d(u) + d(v) - |W| \ge n+1 - |W|,$$
which means that any two not connected vertices in $G$ have at least $3$ common neighbors. 

From Lemma \ref{lem:odd_n+2} we have that there exist two not connected vertices $u$ and $v$ such that $d(u) + d(v) = n+1$. In particular, it means that $u$ and $v$ have exactly $3$ common neighbors. Since $u$ and $v$ see together all the colors, there is exactly one common color $c$ that is seen by both $u$ and $v$. 
We pick vertices $u$ and $v$ in such a way that their sum of the degrees is exactly $n+1$ and their common color $c$ is seen by the smallest number of vertices among all such pairs. 

Let us denote by $A$ the set of colors seen by $u$ but not by $v$, by $A'$ the set of colors seen by $v$ but not by $u$, and let $a = |A|$. Then $A \cup A' \cup \{c\} = C$ and $|A'| = n-a-1$. 
Let $L$ be the set of all the vertices except $v$ and its neighbors and $R$ be the set of all the vertices except $u$ and its neighbors. Then we have $|L| = a-1$ and $|R| = n-a-2$. Vertices in $L$ are not connected with $v$, and so they see all the colors from $A$ and from Claim~\ref{cla:odd_2n-1} they form a clique. Similarly, the vertices in $R$ see all the colors from~$A'$ and form a clique.

\begin{cla}\label{cla:LRedges}
There are at most $\min\{2n+a-5, 3n-a-6\}$ edges outside $L$ and $R$.
\end{cla}

\begin{proof}
    Assume the contrary. Together with $\binom{a-1}{2} + \binom{n-a-2}{2}$ edges inside $L$ and inside $R$ we obtain at least $a^2 - (n-2)a + \frac{1}{2}n^2 - \frac{1}{2}n$ and $a^2 - an + \frac{1}{2}n^2 + \frac{1}{2}n - 1$ edges, respectively. Each bound is minimized for $a = \frac{n-1}{2}$ and equals to $\frac14 n^2 + \frac12 n - \frac34$, contradicting the desired number of edges.
\end{proof}

\begin{cla}\label{cla:Wedges}
There are at most $2|L|+1$ edges between $L$ and $W$, and at most $2|R|+1$ edges between $R$ and $W$. 
\end{cla}

\begin{proof}
From symmetry, assume that there are at least $2|L|+2=2a$ edges between $L$ and $W$. Every vertex in $R$ is not connected with $u$, and needs to have at least $3$ common neighbors with $u$, so it has at least $3$ edges to $L \cup W$. There is also at least one edge between the vertices in $W$ from Claim~\ref{cla:odd_2n-1}. Altogether we have at least 
$2a + 3(n-a-2) + 1 = 3n-a-5$ edges outside $L$ and $R$, contradicting Claim~\ref{cla:LRedges}.
\end{proof}

If a vertex $w$ in $W$ has a non-neighbor $x$ in $L$ and a non-neighbor $y$ in $R$, then $w$ and $x$ must see all the colors from $A'$, and also $w$ and $y$ must see all the colors from $A$. Thus, $w$, $x$ and $y$ see all the colors from $A$ and $A'$ twice. They also must see the color $c$ at least once, and so, their sum of the degrees is at least $2n-1$, contradicting Claim~\ref{cla:odd_2n-1}. Thus, every vertex in $W$ is fully connected to $L$ or $R$. 

If all three vertices in $W$ are fully connected to $L$ or $R$ we get a contradiction with Claim~\ref{cla:Wedges}. So, we can assume that $W$ consists of vertices $\alpha$ and $\beta$ fully connected to $R$ and a vertex $\gamma$ fully connected to $L$. To avoid contradiction with Claim~\ref{cla:Wedges} vertex $\gamma$ is connected with only one vertex in~$R$, with vertex $v$. 

If $\gamma$ is not connected with $\alpha$ and $\beta$, then there exists a color from $A$ that is not seen by $\gamma$, thus this color needs to be seen by all the vertices in $R\setminus\{v\}$ and by $\alpha$ and $\beta$, which means that it is seen by exactly $n-2$ vertices. This is a contradiction with the parity of $n$. So, we can assume that $\gamma$ is connected with $\beta$. 

Every vertex in $L$ has at least $3$ neighbors in $W \cup R$, which together with edges between $W$ and $R$ and the edge $\beta\gamma$ gives $2n+a-5$ edges outside $L$ and $R$. To avoid contradiction with Claim~\ref{cla:LRedges}, every vertex in $L$ needs to have exactly $3$ neighbors in $W \cup R$. Moreover, there are no more edges inside $W$, in particular $\alpha$ and $\beta$ are not connected. 

Every vertex in $L$ needs to be connected with $\alpha$ or $\beta$, as otherwise they crate an independent set of size $3$ contradicting Claim~\ref{cla:odd_2n-1}. Also, every vertex in $L$ except $u$ cannot be connected with both $\alpha$ and $\beta$, because then we have a contradiction with Claim~\ref{cla:Wedges}. So, every vertex in $L\setminus\{u\}$ has exactly one neighbor in $R$. Similarly, every vertex in $R\setminus\{v\}$ has exactly one neighbor in~$L$. Thus, $|L|=|R|=\frac{n-3}{2}$ and there is a perfect matching between the vertices in $L\setminus\{u\}$ and $R\setminus\{v\}$. This is depicted on Figure~\ref{fig:odd_n+1}.

\begin{figure}[ht]
\begin{center}
\includegraphics[scale=1.2]{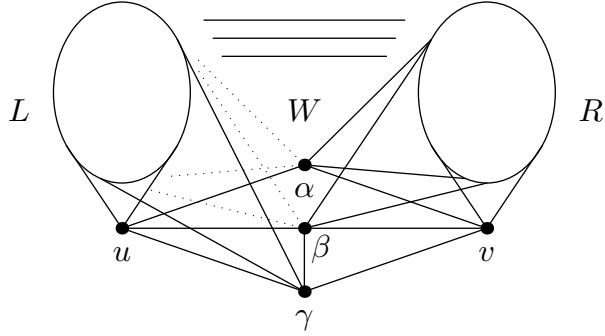}
\end{center}
\caption{A graph of odd order $n$ with two vertices $u$ and $v$ having the sum of the degrees equal to $n+1$. Dotted lines denote the possible edges.}\label{fig:odd_n+1}
\end{figure}

If $\beta$ is fully connected to $L$, then $\alpha$ has only one neighbor in $L$, and $d(\alpha) + d(\gamma) = n$, contradicting Claim~\ref{cla:odd_n+1}. Thus, there exists a vertex $x\in L$ not connected with $\beta$. Let $y$ be the vertex in $R$ connected with $x$.

If $n-1$ colors are seen by the vertices in $W$ at least two times, then we have a contradiction with Claim~\ref{cla:odd_2n-1}. If a color is seen by the vertices in $W$ only once, then it needs to be seen by $\alpha$ and not seen by $\beta$ and $\gamma$, because only $\beta$ and $\gamma$ are connected.  
If any two colors from $A'\cup\{c\}$ are not seen by $\beta$, then the vertex $x$ needs to see those two colors, but it sees also $\frac{n-1}{2}$ colors from $A$, so it has $4$ neighbors outside $L$, which gives a contradiction. Similarly, if any two colors from $A\cup\{c\}$ are not seen by $\gamma$, then $y$ needs to see those two colors, which gives a contradiction the same way. 
Hence, the only way to avoid contradiction, is that there is exactly one color $d \in A\cup\{c\}$ and exactly one color $d' \in A'\cup\{c\}$ that are seen by $\alpha$ and not by $\beta$ and $\gamma$. Moreover, it cannot hold $d=d'=c$. 
Notice that colors $d$ and $d'$ need to be seen by both $x$ and $y$ and any other color in $A \cup A'$ is seen by at least one of them. So, if $d$ or $d'$ is $c$, or if $c$ is seen by $x$ or $y$, then $d(x)+d(y) > n+1$, which gives a contradiction, since $d(x)=d(y)=\frac{n+1}{2}$. Thus, $d \in A$, $d' \in A'$ and $x$ and $y$ do not see color $c$. 

If there is any other vertex in $L$ or $R$ that does not see color $c$, then it needs to be connected with both $x$ and $y$, which gives a contradiction. Thus, color $c$ is seen by exactly $n-3$ vertices. 
Notice that $x$ and $v$ are not connected, $d(x) + d(v) = n+1$ and the only common color of $x$ and $v$ is $d' \in A'$. Recall that the color $c$ was chosen to be the rarest color that appear in such pairs, and so at least $n-3$ vertices see color $d'$. Similarly, by considering vertices $u$ and $y$, we obtain that at least $n-3$ vertices see color $d$. 
Any other color is seen by at least $\frac{n+1}{2}$ vertices, so the total number of edges is at least
$$3\frac{n-3}{2} + (n-3)\frac{n+1}{4} = \frac14 n^2 + n -\frac{21}{4},$$ 
which gives a contradiction with the assumed number of edges for every $n \ge 9$. 
This finishes the proof of Theorem \ref{thm:odd}.
\end{proof}

\end{document}